\documentclass [12pt, leqno]{amsart}
\usepackage{amsmath,amstext,amsthm,amssymb,amsxtra}
\usepackage{mathtools}
\usepackage{enumerate}
\usepackage[colorlinks,citecolor=red,pagebackref,hypertexnames=false]{hyperref}
\usepackage[msc-links,nobysame]{amsrefs}

\usepackage{amsfonts}
\usepackage{amssymb}
\usepackage{amsmath}
\usepackage{graphicx}
\usepackage{graphicx,color}
\usepackage[normalem]{ulem}
\usepackage{todonotes}
\makeatletter
\@namedef{subjclassname@2010}{%
  \textup{2010} Mathematics Subject Classification}
\makeatother

\def\R{\mathbb R}

\def\norm#1.#2.{\lVert#1\rVert_{#2}}
\def\Norm#1.#2.{\bigl\lVert#1\bigr\rVert_{#2}}
\def\NOrm#1.#2.{\Bigl\lVert#1\Bigr\rVert_{#2}}
\def\NORm#1.#2.{\biggl\lVert#1\biggr\rVert_{#2}}
\def\NORM#1.#2.{\Biggl\lVert#1\Biggr\rVert_{#2}}

\def\ip#1,#2,{\langle #1,#2\rangle}
\def\Ip#1,#2,{\bigl\langle#1,#2\bigr\rangle}
\def\IP#1,#2,{\Bigl\langle#1,#2\Bigr\rangle}

\def\XXint#1#2#3{{\setbox0=\hbox{$#1{#2#3}{\int}$}
     \vcenter{\hbox{$#2#3$}}\kern-.5\wd0}}

\newtheorem{thm}{Theorem}
\newtheorem{lem}{Lemma}

\theoremstyle{definition}
\newtheorem*{problem}{Problem}
\newtheorem{example}{Example}
\newtheorem{conjecture}{Conjecture}
\newtheorem*{remark}{Remark}
\begin{document}

\baselineskip=17pt

\title[Solyanik Estimates in Harmonic Analysis]{Solyanik Estimates in Harmonic Analysis}

\author{Paul Hagelstein}
\address{Department of Mathematics, Baylor University, Waco, Texas 76798}
\email{paul\!\hspace{.018in}\_\,hagelstein@baylor.edu}
\thanks{P. H. is partially supported by a grant from the Simons Foundation (\#208831 to Paul Hagelstein).}

\author{Ioannis Parissis}
\address{Department of Mathematics, Aalto University, P. O. Box 11100, FI-00076 Aalto, Finland}
\email{ioannis.parissis@gmail.com}
\thanks{I. P. is supported by the Academy of Finland, grant 138738.}


\begin{abstract}
Let $\mathcal{B}$ denote a collection of open bounded sets in $\mathbb{R}^n$, and define the associated maximal operator $M_{\mathcal{B}}$ by
$$
M_{\mathcal{B}}f(x)  \coloneqq \sup_{x \in R \in \mathcal{B}} \frac{1}{|R|}\int_R |f|.
$$
The \emph{sharp Tauberian constant of $M_{\mathcal{B}}$ associated to $\alpha$}, denoted by $C_{\mathcal{B}}(\alpha)$, is defined as
$$
C_{\mathcal{B}}(\alpha) \coloneqq \sup_{E :\, 0 < |E| < \infty}\frac{1}{|E|}\big|\big\{x \in \mathbb{R}^n:\, M_{\mathcal{B}}\chi_E (x) > \alpha\big\}\big|.$$
Motivated by previous work of A. A. Solyanik, we show that if $M_{\mathcal{B}}$ is the uncentered Hardy-Littlewood maximal operator associated to balls, the estimate
$$
\lim_{\alpha \rightarrow 1^-}C_{\mathcal{B}}(\alpha) = 1
$$
holds.  Similar results for iterated maximal functions are obtained, and open problems in the field of Solyanik estimates are also discussed.
\end{abstract}

\maketitle

\section{Introduction} Let $\mathcal{B}$ denote a collection of open sets in $\mathbb{R}^n$, and define the associated geometric maximal operator $M_{\mathcal{B}}$ by
$$
M_{\mathcal{B}}f(x) \coloneqq \sup_{x \in R \in \mathcal{B}} \frac{1}{|R|}\int_R |f|.
$$

For some examples, if $\mathcal{B}$ were the collection of all cubes or balls in $\mathbb{R}^n$, $M_\mathcal{B}$ would be the uncentered Hardy-Littlewood maximal operator; while if $\mathcal{B}$ were the collection of all rectangles in $\mathbb{R}^n$ with sides parallel to the axes, $M_\mathcal{B}$ would be the strong maximal operator. Due to the importance of these  classical operators we adopt the special notation $M_{\operatorname {HL} } $ for the uncentered Hardy-Littlewood maximal operator and $M_{\operatorname S}$ for the strong maximal operator. To avoid confusion, we will sometimes let $M_{\operatorname {HL},b } $ denote the uncentered Hardy-Littlewood maximal operator $M_{\operatorname{HL}}$ with respect to balls,  and let $M_{\operatorname {HL} ,c} $ denote the uncentered Hardy-Littlewood maximal operator with respect to cubes.

We will also consider the \emph{centered} Hardy-Littlewood maximal operator defined as
$$
M_{\operatorname {HL},b } ^{\operatorname c} f(x) \coloneqq \sup_{r>0}\frac{1}{|B(x,r)|}\int_{B(x,r)} |f|,
$$
where $B(x,r)$ denotes the Euclidean ball of radius $r>0$, centered at $x\in\mathbb R^n$. A similar definition gives $M_{\operatorname {HL},c } ^{\operatorname c}$, defined with respect to centered cubes. Observe that, strictly speaking, these centered operators does not fall under the scope of our general definition for $M_{\mathcal B}$ as there is no collection $\mathcal B$ that will generate $M_{ \operatorname{HL},b } ^{\operatorname c}$ or $M_{ \operatorname{HL},c} ^{\operatorname c}$. This is essentially due to the centered nature of the sets defining $M_{\operatorname{HL}} ^{\operatorname c}$.

Given a collection $\mathcal{B}$ as above, we are typically interested in determining if the associated maximal operator $M_\mathcal{B}$ is bounded on $L^p(\mathbb{R}^n)$ for some \mbox{$1 < p < \infty$} and also what are  the optimal weak type $(p,p)$ estimates that $M_\mathcal{B}$ satisfies. For instance, it is well known that the uncentered Hardy-Littlewood maximal operator $M_{\operatorname{HL}}$ is bounded on $L^p(\mathbb{R}^n)$ for all $1 < p \leq \infty$ and that it satisfies the weak type $(1,1)$ estimate:
$$
| \{x\in\mathbb R^n:\, M_{\operatorname{HL}}f(x) > \alpha \} | \leq \frac{3^n}{\alpha} \|f \|_1.
$$

Even weaker conditions on geometric maximal operators are so-called \emph{Tauberian conditions}.   The maximal operator $M_{\mathcal{B}}$ is said to satisfy a \emph{Tauberian condition with respect to $\alpha\in(0,1)$} if there is some constant $C$ such that
$$
| \{x\in\mathbb R^n:\, M_{\mathcal{B}}\chi_E (x) > \alpha \}| \leq C|E|
$$
holds for all measurable sets $E$. Note that the previous condition is only supposed to hold for some \emph{fixed} $\alpha\in(0,1)$.  Now, if $M_{\mathcal{B}}$ is known to satisfy a weak type $(1,1)$ estimate or to be bounded on $L^{p}$ for some $1 < p < \infty$, then it is easily seen that $M_{\mathcal{B}}$ must satisfy a Tauberian condition with respect to $\alpha$, for \emph{all} $0 < \alpha < 1$.   However, a maximal operator $M_{\mathcal{B}}$ can in fact satisfy a Tauberian condition with respect to some $0 < \alpha < 1$ without being $L^p$ bounded for any finite $p$.   A quick example of this type of behavior can be exhibited by, say, letting $\mathcal{B}$ be the collection of all sets of the form $[0,1] \cup(x, x+2)$ and observing that, while $M_{\mathcal{B}}$ satisfies a Tauberian condition with respect to 4/5, it is not bounded on $L^p(\mathbb{R})$ for any $1 < p < \infty$.

A Tauberian condition on a maximal operator, although quite weak, is still very useful, as was shown by A. C\'ordoba and R. Fefferman in their work \cite{CorF} relating the $L^{p}$ bounds of certain multiplier operators to the weak type $\big((\frac{p}{2})', (\frac{p}{2})'\big)$ bounds of associated geometric maximal operators; see \cite{CorF} for  details. Moreover, Hagelstein and Stokolos have shown in \cite{hs} that, provided $\mathcal{B}$ is a homothecy invariant basis of convex sets in $\mathbb{R}^n$, if $\mathcal{B}$ satisfies a Tauberian condition with respect to \emph{some} $0 < \alpha < 1$, then $M_{\mathcal{B}}$ must be bounded on $L^p(\mathbb{R}^n)$ for sufficiently large $p$. This work has recently been extended by Hagelstein, Luque, and Parissis in \cite{hlp} to yield weighted $L^p$ bounds on  maximal operators satisfying a Tauberian condition with respect to a weighted basis.

The issue of \emph{sharp Tauberian constants} is one that has received very little attention until recently.  For specificity, given a maximal operator $M_{\mathcal{B}}$, we define the Tauberian constant $C_{\mathcal{B}}(\alpha)$ by

\begin{equation}\label{e.CBalpha}
C_{\mathcal{B}}(\alpha) \coloneqq \sup_{ E\subset \R^n  :\, 0 < |E| < \infty}\frac{1}{|E|} | \{x \in \mathbb{R}^n:\, M_{\mathcal{B}}\chi_E (x) > \alpha \} |.
\end{equation}

We note here that, in the relevant literature, the function $\phi_{\mathcal B}:[1,\infty)\to \mathbb R$ defined as
$\phi_{\mathcal B}(\lambda)\coloneqq C_{\mathcal{B}}(1/\lambda)$, $\lambda>1$, is many times called the \emph{Halo function} of the collection $\mathcal B$, as for example in \cite{Gu}. Obviously, it is equivalent to study the function $C_{\mathcal B} (\alpha)$ for $\alpha<1$ which is the setup we adopt in this paper.

We will use the special notation $C_{\operatorname {HL},b}$, $C_{\operatorname {HL},c}$ and $C_{\operatorname S}$ for the sharp Tauberian constants corresponding to the basis of balls, cubes, and  axes parallel rectangles, respectively. For the centered Hardy-Littlewood maximal operator we denote the corresponding sharp Tauberian constants by $C_{\operatorname{HL},b}  ^{\operatorname c}$ and $C_{{\operatorname{HL},c} } ^{\operatorname c}$.

Now, if the maximal operator $M_{\mathcal{B}}$ satisfies a weak type $(1,1)$ estimate
$$
 | \{x \in \mathbb{R}^n : \, M_{\mathcal{B}}f(x) > \alpha \} | \leq \frac{C}{\alpha} \|f \|_{1},
$$
then the associated sharp Tauberian constant $C_{\mathcal{B}}(\alpha)$ must satisfy
$$
C_{\mathcal{B}}(\alpha) \leq \frac{C}{\alpha}.
$$
However, we might expect in many situations $C_{\mathcal{B}}(\alpha)$ to be significantly smaller than $C/ \alpha$. For example, even though the weak type $(1,1)$ bound of the uncentered Hardy-Littlewood maximal operator $M_{\operatorname{HL}}$ acting on functions on $\mathbb{R}$ is 2, we would suspect it unlikely to find a set $E$ contained in $\mathbb{R}$ such that $ | \{x \in \mathbb{R} :\, M_{\operatorname{HL}}\chi_E(x) > .99 \} | = 2|E|$.   We will show momentarily that this indeed cannot be the case, and in fact that we must have
\[
\lim_{\alpha \rightarrow 1^-} C_{\operatorname{HL}}(\alpha)=\lim_{\alpha \rightarrow 1^-} \sup_{E \subset \mathbb{R}^n :\, 0<|E| < \infty}\frac{1}{|E|}\big|\big\{x\in\mathbb R^ n :\, M_{\operatorname{HL}}\chi_E(x) > \alpha\big\}\big| = 1 .
\]

The first estimates along the lines of the one above were obtained by \mbox{A. A.} Solyanik in \cite{Solyanik}.  In his honor, we call a result of the form
\[
\lim_{\alpha \rightarrow 1^-}C_{\mathcal{B}}(\alpha) = 1
\]
a \emph{Solyanik estimate}.

\begin{thm}[Solyanik]\label{t.solyanik} We have the following Solyanik estimates:
\[
\lim_{\alpha \rightarrow 1^-} C_{\operatorname{HL},c}(\alpha) =1\quad\text{and}\quad \lim_{\alpha \rightarrow 1^-} C_{\operatorname{S}}(\alpha) =1.
\]
In particular,
\[
 C_{\operatorname{HL},c}(\alpha)-1\sim_n \big(\frac{1}{\alpha}-1\big)^\frac{1}{n}\quad \text{and}\quad  C_{\operatorname{S}}(\alpha)-1\sim_n \big(\frac{1}{\alpha}-1\big)^\frac{1}{n}.
\]
For the sharp Tauberian constant of the centered Hardy-Littlewood maximal operator (with respect to cubes or balls) we have
\[
\lim_{\alpha \rightarrow 1^-} C_{\operatorname{HL},b}  ^{\operatorname c} (\alpha)=1\quad ;\quad\lim_{\alpha \rightarrow 1^-} C_{\operatorname{HL}, c}  ^{\operatorname c} (\alpha)=1
\]
and in particular
\[
C_{\operatorname{HL},b} ^{\operatorname c} (\alpha)-1\sim_n \frac{1}{\alpha}-1\quad;\quad C_{\operatorname{HL},c} ^{\operatorname c} (\alpha)-1\sim_n \frac{1}{\alpha}-1
\]
as $\alpha\to 1^{-}$.
\end{thm}

\begin{remark} Note that Solyanik's theorem does not include an estimate for $C_{\operatorname{HL},b}$ associated to  the uncentered Hardy-Littlewood maximal operator with respect to balls, $M_{\operatorname{HL},b}$.   Indeed, Solyanik concludes the estimate for $C_{\operatorname{HL},c }$ as a corollary of estimate for $C_S$ and thus the methods in his paper do not readily apply to non-centered maximal operators defined with respect to balls. However, the method of Solyanik for centered maximal operators deals equally well with balls or cubes. This is because the basic underlying ingredient for these estimates in the case of centered operators is the Besicovitch covering lemma which works equally well for balls or cubes.
\end{remark}

We now introduce the directional maximal operator $M_{j}$, $j=1,\ldots,n$, acting on $\mathbb R^n$ and defined by
$$M_{j}f(x_1, \ldots, x_n) \coloneqq \sup_{s < x_j < t}\frac{1}{t-s}\int_{s}^{t}|f(x_1, \ldots, x_{j-1}, u, x_{j+1}, \ldots, x_n)|\,du.$$
In the next section we will prove a Solyanik estimate for the iterated maximal operator $M_1\cdots M_n$, namely that we have
$$
\lim_{\alpha \rightarrow 1^-} \sup_{E \subset \mathbb{R}^n :\, 0<|E| < \infty}\frac{1}{|E|}\big|\big\{x \in \mathbb{R}^n:\, M_1\cdots M_n\chi_E(x) > \alpha\big\}\big|  = 1.
$$
This will be done by proving a Solyanik estimate for $M_{\operatorname{HL}}$ on $\mathbb{R}^{1}$ by a  means different than Solyanik did in \cite{Solyanik} but one enabling us to afterwards apply induction to get the desired estimate for $M_{1}\cdots M_{n}$.   Subsequently we will provide a Solyanik estimate for the uncentered maximal operator $M_{\operatorname{HL}, b}$ by utilizing the circle of ideas developed by A. C\'ordoba and R. Fefferman in their work \cite{CorF75} relating covering lemmas to weak type bounds of geometric maximal operators.   Afterwards we will visit the issue of generalizing Solyanik estimates to encompass maximal operators $M_{\mathcal{B}}$ where $\mathcal{B}$ is a homothecy invariant collection of convex sets.  Throughout this paper we will indicate open problems and directions for further research.

\subsection*{Notation} We write $A\lesssim B$ whenever there is a numerical constant $c>0$ such that $A\leq c  B$. We also write $A\sim B$ if $A\lesssim B$ and $B\lesssim A$. If the constant $c$ depends for example on the dimension $n$ we will write $A\sim_n B$.

\section{Solyanik Estimates for Iterated Maximal Functions}

The key result in this section is the following lemma.

\begin{lem}\label{l.oned}
Let $M_{\operatorname{HL}}$ denote the uncentered Hardy-Littlewood maximal operator acting on functions on $\mathbb{R}$.    Let $E \subset \mathbb{R}$, where $|E| < \infty$, and let $0 \leq \gamma < \alpha < 1$.   Then
\begin{equation}\label{e1}
 | \{x \in \mathbb{R} :\, M_{\operatorname{HL}}(\chi_{E} + \gamma \chi_{E^\mathtt{c}})(x) > \alpha \} | \leq \bigg(1 + 4 \frac{1-\alpha}{\alpha - \gamma}\bigg) |E |.
\end{equation}
\end{lem}

\begin{proof} We first prove the lemma for $\gamma>0$. Let $f_{E, \gamma}$ be the function defined on $\mathbb{R}$ by
$$
f_{E, \gamma}(x) = \chi_E(x) + \gamma\chi_{E^\mathtt{c}}(x) .
$$
Let $ \{I_{j} \}$ be a countable collection of intervals such that
$$
\{x  \in \mathbb{R}:\, M_{\operatorname{HL}}(\chi_{E} + \gamma \chi_{E^\mathtt{c}})(x) > \alpha \} = \cup_j I_j
$$
and such that, for each $j$,
$$
\frac{1}{|I_j|}\int_{I_j}f_{E, \gamma}  >  \alpha .
$$
We now fix some $\epsilon > 0$. As $M_{\operatorname{HL}}$ is of weak type $(1,1)$ we must have that
$$
\{x  \in \mathbb{R}: \, M_{\operatorname{HL}}(\chi_{E} + \gamma \chi_{E^\mathtt{c}})(x) > \alpha\}
$$
is of finite measure. Indeed, if $ M_{\operatorname{HL}}(\chi_{E} + \gamma \chi_{E^\mathtt{c}})(x) > \alpha$ then we must have $M_{\operatorname{HL}}\chi_E(x) > \alpha - \gamma$. Accordingly there exists a finite subcollection $\{I_{j}' \}$ of $ \{I_{j} \}$ such that
$$
| \{x  \in \mathbb{R} :\, M_{\operatorname{HL}}(\chi_{E} + \gamma \chi_{E^\mathtt{c}})(x) > \alpha \} \setminus \cup_j I_{j}' | < \epsilon .
$$
Arguing as in \cite{baf}*{p. 24} we see that there exists a collection of intervals $\{\tilde{I}_j\}_j$ contained in $ \{I'_j \}_j$ such that
$\cup_j \tilde{I}_j = \cup_j I'_j$ and  $\sum_j \chi_{\tilde{I}_j} \leq 2$. Since $\frac{1}{|\tilde{I}_j|} \int_{\tilde{I}_j}f_{E, \gamma} > \alpha$, we have
$$
|E \cap \tilde{I}_j| + \gamma|\tilde{I}_j \setminus E| > \alpha |\tilde{I}_j|,
$$
implying
$$
\frac{|E \cap \tilde{I}_{j}|}{|\tilde{I}_{j}|} > \frac{\alpha - \gamma}{1 - \gamma}.
$$
So
\[
\begin{split}
\big|\big\{x \in \mathbb{R} :\, M_{\operatorname{HL}}f_{E, \gamma}(x) > \alpha\big\}\big| &\leq |E| + \frac{1 - \alpha}{1 - \gamma}\sum|\tilde{I}_j| + \epsilon \notag
\\
&\leq |E| + 2 \frac{1 - \alpha}{1 - \gamma}|\cup \tilde{I}_j| + \epsilon.
\end{split}
\]
As we have shown that $\frac{1}{|\tilde{I}_j|}\int_{\tilde{I}_j} f_{E, \gamma} > \alpha$ implies
$$
\frac{|E \cap \tilde{I}_j|}{|\tilde{I}_{j}|} > \frac{\alpha - \gamma}{1 - \gamma},
$$
we have
$$
\cup_j \tilde{I}_{j} \subset \bigg\{x  \in \mathbb{R}:\, M_{\operatorname{HL}}\chi_{E}(x) > \frac{\alpha - \gamma}{1 - \gamma}\bigg\}.
$$
So by the weak type $(1,1)$ bound of 2 of $M_{\operatorname{HL}}$ on $\mathbb{R}^{1}$, we have
$$
|\cup \tilde{I}_{j}| \leq 2 \frac{1 - \gamma}{\alpha - \gamma}|E|
$$
and accordingly
$$
\big|\big\{x  \in \mathbb{R}:\, M_{\operatorname{HL}}f_{E, \gamma}(x) > \alpha\big\}\big| \leq \bigg(1 + 4\frac{1 - \alpha}{\alpha - \gamma}\bigg)|E| + \epsilon.$$
As $\epsilon > 0$ was arbitrary we obtain the desired result in the case $\gamma>0$.

Now observe that for any $\alpha,\delta>0$ we have
$$
|\{x\in \mathbb R:\, M_{\operatorname{HL}}(\chi_E)>\alpha \}|\leq |\{x\in\mathbb R:\, M_{\operatorname{HL}}f_{E,\delta}>\alpha\}|\leq  \bigg(1 + 4\frac{1 - \alpha}{\alpha - \delta}\bigg)|E|
$$
by the case already proved. Since the left hand side of the estimate above does not depend on $\delta$ we can let $\delta\to 0^+$ to get the lemma for $\gamma=0$ as well.
\end{proof}

We now iterate the above estimate to yield a Solyanik estimate for the iterated maximal operator $M_{1}\cdots M_{n}$.

\begin{lem}\label{l.iterated}
Setting $\alpha_{0} = 0$ and  $0 < \alpha_{1} < 1$, define $\alpha_j$, $j= 2, 3, 4, \ldots, n$ by
$$
\alpha_j = 1 -  (1 - \alpha_1 )^j .
$$
Then
$$
\big|\big\{x \in \mathbb{R}^{n} :\, M_{1}\cdots M_{n}\chi_E(x) > \alpha_n\big\}\big| \leq \bigg(1 + 4\frac{1 - \alpha_1}{\alpha_1}\bigg)^n |E |$$
holds for every measurable set $E$ in $\mathbb{R}^{n}$.
\end{lem}

\begin{proof}
We proceed by proving
$$
 | \{x \in \mathbb{R}^{n} : \, M_{1}\cdots M_{N}\chi_E(x) > \alpha_N \} | \leq \bigg(1 + 4\frac{1 - \alpha_1}{\alpha_1}\bigg)^N |E | , \quad  N = 1, \ldots, n,
$$
by induction on $N$.  Note
$$
 | \{x \in \mathbb{R}^n : \, M_{1}\chi_E(x) > \alpha_1 \} | \leq \bigg(1 + 4 \frac{1 - \alpha_1}{\alpha_1}\bigg)|E|
$$
holds by Lemma~\ref{l.oned}, seen by setting $\alpha = \alpha_1$, $\gamma = 0$.

Suppose now
$$
 | \{x \in \mathbb{R}^{n} : \, M_1 \cdots M_j \chi_{E}(x) > \alpha_j \} | \leq \bigg(1 + 4\frac{1 - \alpha_1}{\alpha_1}\bigg)^j |E |.
$$
Let
$$
E_j \coloneqq \{x \in \mathbb{R}^n :\, M_1 \cdots M_j\chi_E(x) > \alpha_j \}.
$$
Observe that the $\alpha_j$ satisfy
$$
\frac{1 - \alpha_{j+1}}{\alpha_{j+1} - \alpha_j} = \frac{1 - \alpha_j}{\alpha_j - \alpha_{j-1}},
$$
implying
$$
\frac{1 - \alpha_{j+1}}{\alpha_{j+1} - \alpha_j} = \frac{1 - \alpha_j}{\alpha_j - \alpha_{j-1}} = \cdots = \frac{1 - \alpha_1}{\alpha_1}.
$$
Also, for any $j$ we have
\[
\begin{split}
M_{j+1}M_{1}\cdots M_j\chi_E (x) &=M_{j+1}(\chi_{E_j}M_{1}\cdots M_j\chi_E+\chi_{E_j ^\mathtt{c}} M_{1}\cdots M_j\chi_E)(x)
\\
& \leq M_{j+1}(\chi_{E_j} + \alpha_j \chi_{E_j ^\mathtt{c}})(x).
\end{split}
\]
Hence
\begin{align*}
&\big|\big\{x  \in \mathbb{R}^{n} :\, M_{j+1}M_{1}\cdots M_j\chi_E (x) > \alpha_{j+1}\big\}\big|
\\
&\leq \big|\big\{x  \in \mathbb{R}^{n} :\, M_{j+1} (\chi_{E_j} + \alpha_j \chi_{E_{j}^\mathtt{c}} )(x) > \alpha_{j+1}\big\}\big|
\\
&\leq \bigg(1 + 4\frac{1 - \alpha_{j+1}}{\alpha_{j+1} - \alpha_j}\bigg)|E_j|\quad\quad\text{(by Lemma~\ref{l.oned})}
\\
&\leq \bigg(1 + 4\frac{1 - \alpha_1}{\alpha_1}\bigg)\bigg(1 + 4\frac{1 - \alpha_1}{\alpha_1}\bigg)^{j}|E|
\\
&\leq \bigg(1 + 4\frac{1 - \alpha_{1}}{\alpha_1}\bigg)^{j+1}|E|.
\end{align*}
Since this holds for every measurable set $E$ in $\mathbb{R}^n$, by symmetry we have
$$
\big|\big\{x \in \mathbb{R}^n :\, M_{1}\cdots M_{j+1}\chi_E(x) > \alpha_{j+1}\big\}\big| \leq \bigg(1 + 4\frac{1 - \alpha_1}{\alpha_1}\bigg)^{j+1}|E|.
$$
Setting $j = n - 1$ yields the desired result.
\end{proof}

\begin{thm}\label{t.iterated}
Let $0 < \alpha < 1$.  Then
$$
\big|\big\{x \in \mathbb{R}^n :\, M_1\cdots M_n\chi_E(x) > \alpha \big\}\big| \leq \bigg(1 + 4 \frac{(1 - \alpha)^{1/n}}{1 - (1 - \alpha)^{1/n}}\bigg)^n |E| .
$$
Accordingly, letting $C_{1\cdots n}(\alpha)$ denote the sharp Tauberian constant with respect to $\alpha$ of $M_1\cdots M_n$, we have
$$ C_{1\cdots n}(\alpha) - 1 \sim_n (\frac{1}{\alpha} - 1)^{1/n} .$$
\end{thm}

\begin{proof}
Using the notation of the previous lemma, we let $\alpha_n = \alpha$.   The corresponding $\alpha_1$ satisfies
$$
\alpha = 1 - (1 - \alpha_1)^n,
$$
implying that
$$
\alpha_1 = 1 -  (1 - \alpha)^{1/n}.
$$
The result follows by Lemma~\ref{l.iterated}.
\end{proof}

\section{Solyanik Estimates for the Uncentered Hardy-Littlewood maximal operator}

The primary goal in this section is to provide a Solyanik estimate for the uncentered Hardy-Littlewood maximal operator $M_{\operatorname{HL}, b}$.
\begin{thm}\label{un.solyanik}
Let $M_{\operatorname{HL}, b}$ denote the non-centered Hardy-Littlewood maximal operator, defined with respect to balls in $\mathbb{R}^n$.  Then we have the corresponding Solyanik estimate
$$
\lim_{\alpha \rightarrow 1^{-}} C_{\operatorname{HL}, b}(\alpha) = 1\;.$$
In particular we have that
$$
C_{\operatorname{HL}, b}(\alpha) - 1 \lesssim_n \big( \frac{1}{\alpha} - 1 \big)^{\frac{1}{n+1}}
$$
as $\alpha \rightarrow 1^{-}$.
\end{thm}
\begin{proof}
Let $0 < \alpha < 1$, and let $E$ be a set of finite measure in $\mathbb{R}^n$.  Let $\{B_j\}$ be a collection of balls such that
$$ \left\{x \in \mathbb{R}^n : M_{\operatorname{HL}, b}\chi_E (x) > \alpha\right\} = \cup_j B_j,$$
where every $B_j$ satisfies
$$\frac{1}{|B_j|} \int_{B_j}\chi_E > \alpha.$$
Without loss of generality we may assume that $\{B_{j}\}_j$ is a finite collection $\{B_{j}\}_{j=1}^N$ as our estimates of $|\cup B_j|$ will be independent of $N$. We reorder the balls $B_j$ so that they are nonincreasing in size, i.e.
$$
|B_1| \geq |B_2|  \geq \cdots\geq |B_N|.
$$
We will now obtain a subcollection $\{\tilde{B}_j\}_j$ using a selection algorithm motivated by ideas of A. C\'ordoba and R. Fefferman in \cite{CorF75}.   Let $1 > \delta > 0$; here we think of $\delta$ as being very close to 0. We choose $\tilde{B}_1 = B_1$.   Assume $\tilde{B}_1, \ldots, \tilde{B}_k$ have been selected and suppose that $\tilde B_k=B_M$ for some positive integer $M<N$. We let $\tilde{B}_{k+1}$ be the first $B_j$ on the list $B_{M+1}, B_{M+2}, \ldots,B_N$ such that
$$|B_{j} \cap (\cup_{i=1}^k \tilde{B}_i)| \leq (1-\delta) |B_j|.$$
If such a $B_j$ does not exist, the list of selected balls terminates with $\tilde{B}_k.$

Let now $x \in \{x \in \mathbb{R}^n : M_{\operatorname{HL}, b}\chi_E(x) > \alpha\}$ so $x$ necessarily lies in one of the balls $B_j$. Suppose for the moment that $B_j$ is \emph{not} one of the selected balls.  Let $B_x$ be a ball of volume $\delta |B_j|$ containing $x$ and contained in $B_{j}$.  Since $B_j$ was not selected, $B_x$ must intersect a $\tilde{B}_k$ of size larger than that of $B_j$.  As the radius of $B_x$ is less than $\delta^{1/n}$ times the radius of $\tilde{B}_k$, by the triangle inequality we have $x \in (1 + 2 \delta^{1/n})\tilde{B}_k$, where for a ball $B$ in $\mathbb{R}^n$ we let $cB$ denote the $c$-fold concentric dilate of $B$.   So
$$\{x \in \mathbb{R}^n : M_{\operatorname{HL}, b}\chi_E(x) > \alpha\} \subset \cup_j (1 + 2 \delta^{1/n})\tilde{B}_j.$$
Let now
$$\tilde{E}_j \coloneqq \tilde{B}_j \backslash \cup_{i=1}^{j-1}\tilde{B}_i.$$
We have that
$$\left|\{x \in \mathbb{R}^n : M_{\operatorname{HL}, b}\chi_E(x) > \alpha\}\right| \leq \sum_j (1 + 2 \delta^{1/n})^n|\tilde{E}_j|.$$
Since for each $j$ we have $\frac{1}{|\tilde{B}_j|}\int_{\tilde{B}_j}\chi_E > \alpha$ and moreover $|\tilde{E}_j| / |\tilde{B}_j| > \delta$, we conclude
\[
\begin{split}
\frac{1}{|\tilde{E}_j|}\int_{\tilde{E}_j}\chi_E &\geq \big[ \alpha|\tilde{B}_j| - (|\tilde{B}_j| - |\tilde{E}_j|)\big]/ |\tilde{E}_j|
\\
&\geq 1 - (1-\alpha){|\tilde{B}_j|}/{|\tilde{E}_j|} \geq 1 -(1-\alpha ) \delta^{-1}
\\
&\geq [ \delta - (1 - \alpha)]/ \delta.
\end{split}
\]
Placing an additional  restriction on $\delta$ by requiring that $1 > \delta > 1 - \alpha$, we have
$$
|\tilde{E}_j| < \frac{\delta}{\delta - (1 - \alpha)}|E \cap \tilde{E}_j|.
$$
As the $\tilde{E}_j$ are disjoint, we then have
$$
|\{x \in \mathbb{R}^n : M_{\operatorname{HL}, b}\chi_E(x) > \alpha\}| \leq (1 + 2 \delta^{1/n})^n\frac{\delta}{\delta - (1 - \alpha)}|E|.
$$
Setting $\delta = (1 - \alpha)^{\frac{n}{n+1}}$ then yields the desired estimate.
\end{proof}

We strongly suspect the the bound $( \frac{1}{\alpha} - 1 )^{1/(n+1)}$ is not sharp, as indicated by the following example.

\begin{example}\label{exa.slab} Let $E$ be the $n$-dimensional rectangle
	\[
	E\coloneqq[-100,100]\times\cdots\times[-100,100]\times[-1,1].
	\]
	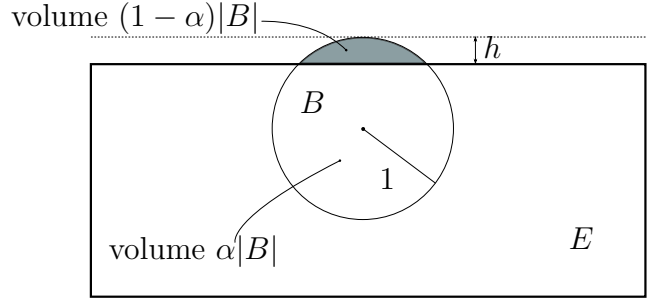
\begin{figure}[htb]
	\centering
	\def\svgwidth{240pt}
	%% Creator: Inkscape inkscape 0.48.2, www.inkscape.org
	%% Accompanies image file 'slab.pdf' (pdf, eps, ps)
	\begingroup%
	  \makeatletter%
	  \providecommand\color[2][]{%
	    \errmessage{(Inkscape) Color is used for the text in Inkscape, but the package 'color.sty' is not loaded}%
	    \renewcommand\color[2][]{}%
	  }%
	  \providecommand\transparent[1]{%
	    \errmessage{(Inkscape) Transparency is used (non-zero) for the text in Inkscape, but the package 'transparent.sty' is not loaded}%
	    \renewcommand\transparent[1]{}%
	  }%
	  \providecommand\rotatebox[2]{#2}%
	  \ifx\svgwidth\undefined%
	    \setlength{\unitlength}{281.27867067bp}%
	    \ifx\svgscale\undefined%
	      \relax%
	    \else%
	      \setlength{\unitlength}{\unitlength * \real{\svgscale}}%
	    \fi%
	  \else%
	    \setlength{\unitlength}{\svgwidth}%
	  \fi%
	  \global\let\svgwidth\undefined%
	  \global\let\svgscale\undefined%
	  \makeatother%
	  \begin{picture}(1,0.4589)%
	    \put(0,0){\includegraphics[width=\unitlength]{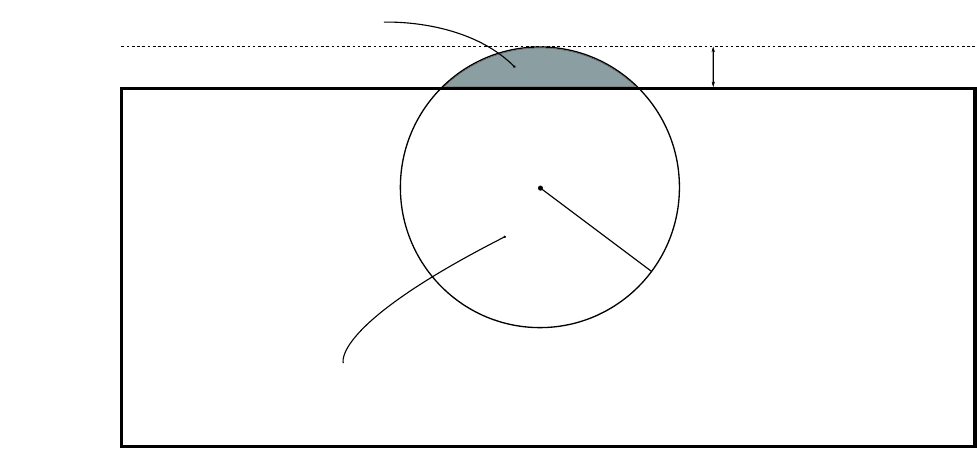}}%
	    \put(0.27049931,0.38422147){\color[rgb]{0,0,0}\makebox(0,0)[lt]{\begin{minipage}{0.09193698\unitlength}\raggedright \end{minipage}}}%
	    \put(-0.00118258,0.42838489){\color[rgb]{0,0,0}\makebox(0,0)[lb]{\smash{volume $(1-\alpha)|B|$}}}%
	    \put(0.74291265,0.37894806){\color[rgb]{0,0,0}\makebox(0,0)[lb]{\smash{$h$}}}%
	    \put(0.15242719,0.09836996){\color[rgb]{0,0,0}\makebox(0,0)[lt]{\begin{minipage}{0.48756935\unitlength}\raggedright volume $\alpha|B|$\end{minipage}}}%
	    \put(0.87766124,0.11146486){\color[rgb]{0,0,0}\makebox(0,0)[lt]{\begin{minipage}{0.18690161\unitlength}\raggedright $E$\end{minipage}}}%
	    \put(0.45357744,0.29125595){\color[rgb]{0,0,0}\makebox(0,0)[lb]{\smash{$B$}}}%
	    \put(0.57902496,0.17241096){\color[rgb]{0,0,0}\makebox(0,0)[lb]{\smash{$1$}}}%
	  \end{picture}%
	\endgroup%
	
	\caption{A ball $B$ intersecting the slab $E$.}
	\end{figure}
Consider a ball $B$ of radius $1$ intersecting the rectangle $E$ on one of its long sides and away from its corners, so that a $(1-\alpha)$ portion of $|B|$ lies outside $E$. One can calculate that the union of all such balls constitutes a region of measure approximately $(1 + h)|E|$ with $h\simeq_n (\frac{1}{\alpha} - 1)^\frac{2}{n+1}$. We conclude $$C_{\operatorname{HL,b}}(\alpha)-1\gtrsim_n(\frac{1}{\alpha} - 1)^\frac{2}{n+1}.$$

In contrast, by doing a similar calculation with a unit cube $Q$ meeting the set $E$ at an angle $\pi/4$ we get $h\simeq_n (\frac{1}{\alpha} - 1)^\frac{1}{n}$. This proves the lower bound
$$ C_{\operatorname{HL},c}(\alpha)-1\gtrsim_n\big(\frac{1}{\alpha} - 1\big)^\frac{1}{n}.$$

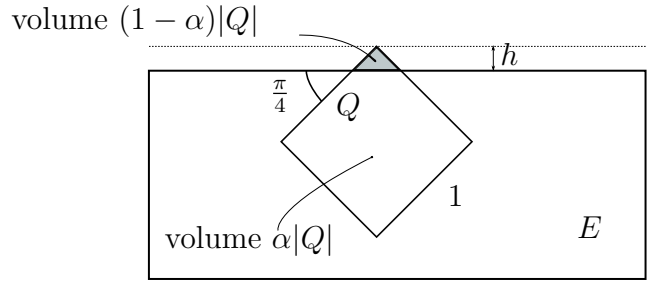
\begin{figure}[htb]
\centering
\def\svgwidth{240pt}
%% Creator: Inkscape inkscape 0.48.2, www.inkscape.org
%% Accompanies image file 'slab_cube.pdf' (pdf, eps, ps)
\begingroup%
  \makeatletter%
  \providecommand\color[2][]{%
    \errmessage{(Inkscape) Color is used for the text in Inkscape, but the package 'color.sty' is not loaded}%
    \renewcommand\color[2][]{}%
  }%
  \providecommand\transparent[1]{%
    \errmessage{(Inkscape) Transparency is used (non-zero) for the text in Inkscape, but the package 'transparent.sty' is not loaded}%
    \renewcommand\transparent[1]{}%
  }%
  \providecommand\rotatebox[2]{#2}%
  \ifx\svgwidth\undefined%
    \setlength{\unitlength}{314.07866692bp}%
    \ifx\svgscale\undefined%
      \relax%
    \else%
      \setlength{\unitlength}{\unitlength * \real{\svgscale}}%
    \fi%
  \else%
    \setlength{\unitlength}{\svgwidth}%
  \fi%
  \global\let\svgwidth\undefined%
  \global\let\svgscale\undefined%
  \makeatother%
  \begin{picture}(1,0.42371159)%
    \put(0,0){\includegraphics[width=\unitlength]{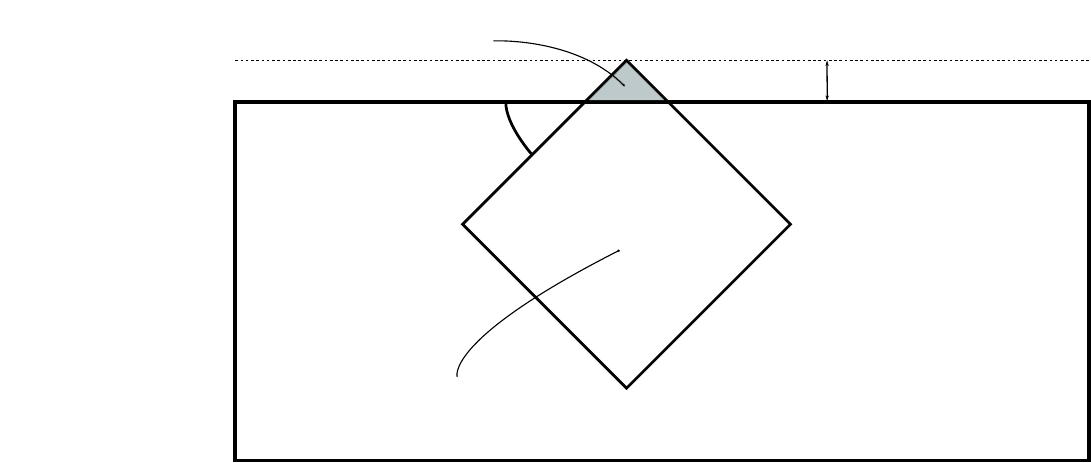}}%
    \put(0.34668284,0.34409629){\color[rgb]{0,0,0}\makebox(0,0)[lt]{\begin{minipage}{0.08233578\unitlength}\raggedright \end{minipage}}}%
    \put(-0.00105908,0.39638324){\color[rgb]{0,0,0}\makebox(0,0)[lb]{\smash{volume $(1-\alpha)|Q|$}}}%
    \put(0.7697609,0.3393736){\color[rgb]{0,0,0}\makebox(0,0)[lb]{\smash{$h$}}}%
    \put(0.24094127,0.08809694){\color[rgb]{0,0,0}\makebox(0,0)[lt]{\begin{minipage}{0.4366513\unitlength}\raggedright volume $\alpha|Q|$\end{minipage}}}%
    \put(0.89043737,0.09982432){\color[rgb]{0,0,0}\makebox(0,0)[lt]{\begin{minipage}{0.16738302\unitlength}\raggedright $E$\end{minipage}}}%
    \put(0.51064167,0.26083938){\color[rgb]{0,0,0}\makebox(0,0)[lb]{\smash{$Q$}}}%
    \put(0.68666673,0.11619869){\color[rgb]{0,0,0}\makebox(0,0)[lb]{\smash{$1$}}}%
    \put(0.40648216,0.28176223){\color[rgb]{0,0,0}\makebox(0,0)[lb]{\smash{$\frac \pi 4$}}}%
  \end{picture}%
\endgroup%
\caption{A cube $Q$ intersecting the slab $E$.}
\end{figure}
\end{example}

Observe that the latter calculation indicates that the Solyanik estimate for iterated maximal functions provided by Theorem 2 is sharp.   Moreover, the fact that the slab example provides a \emph{better} Solyanik estimate for $M_{\operatorname{HL}, b}$ inclines us to believe that Theorem 3 is not sharp, and a more refined argument might prove the following:
\begin{conjecture}
a)   We have the asymptotic estimate
$$
C_{\operatorname{HL}, b}(\alpha) - 1 \sim_n \big(\frac{1}{\alpha}- 1\big)^{\frac{1}{n}}
$$
as $\alpha\to 1^{-}$.   The exponent here is a natural one to consider, as $ (\frac{1}{\alpha}- 1\big)^{\frac{1}{n}}$ is the sharp Solyanik exponent associated to $M_{\operatorname{HL}, c}$ and $M_1 \cdots M_n$.

b)  A stronger asymptotic estimate, motivated by Example~\ref{exa.slab} above, would be that
$$
C_{\operatorname{HL}, b}(\alpha) - 1 \sim_n\big ( \frac{1}{\alpha} - 1\big)^{\frac{2}{n+1}}
$$
as $\alpha\to 1^{-}$.
\end{conjecture}

\section{Solyanik estimates for homothecy invariant bases of convex sets}

With the Solyanik estimates associated to Theorems~\ref{t.solyanik}-\ref{un.solyanik} in hand, it is natural to try to extend these types of results to encompass maximal operators such as the maximal operator with respect to rectangles along lacunary directions. Rather than focus our attention on a particular maximal operator, we will here consider the following more general problem:

\begin{problem} Let $\mathcal{B}$ denote a collection of open  bounded sets in $\mathbb{R}^n$ and $M_{\mathcal{B}}$ the associated geometric maximal operator.  Define the associated Tauberian constants $C_{\mathcal{B}}(\alpha)$ by
$$
C_{\mathcal{B}}(\alpha) \coloneqq \sup_{E :\, 0 < |E| < \infty}\frac{1}{|E|} | \{x \in \mathbb{R}^n:\, M_{\mathcal{B}}\chi_E (x) > \alpha \}|.
$$
For which $\mathcal{B}$ do we have
$$
\lim_{\alpha \rightarrow 1^-}C_{\mathcal{B}}(\alpha) = 1?
$$
\end{problem}

We would expect that the maximal operator $M_{\mathcal{B}}$ should be somewhat well-behaved in order to have $\lim_{\alpha \rightarrow 1^-}C_{\mathcal{B}}(\alpha) = 1$, as such an estimate would not hold if $\mathcal{B}$ were, say, the collection of all rectangles in $\mathbb{R}^2$.
However, simple $L^p$ boundedness of $M_{\mathcal{B}}$ or even weak type (1,1) bound on $M_{\mathcal{B}}$ is not enough to guarantee that $M_{\mathcal{B}}$ satisfies a Solyanik estimate, as is indicated by the following example of Beznosova and Hagelstein found in \cite{BH}.

\begin{example}\label{exa.nonconvex} Let $\mathcal{B}$ consist of all the homothecies of sets in $\mathbb{R}$ in the collection
$$\{((0,1) \cup (x, x+ \epsilon))\cap(0,2):\, x \in (0,2),\, \epsilon > 0\}.$$
The operator $M_{\mathcal{B}}$ is  dominated by twice the Hardy-Littlewood maximal operator and hence is bounded on $L^{p}(\mathbb{R})$ for $1 < p \leq \infty$ and is of weak type $(1,1)$.    Observe however that $M_{\mathcal{B}}\chi_{(0,1)} = 1$ on $(0,2)$ and hence we have that $\lim_{\alpha \rightarrow 1^{-}} C_{\mathcal{B}}(\alpha) \geq 2$.
\end{example}

Note that the sets in the collection $\mathcal{B}$ above are not all \emph{convex}.  We have previously seen convexity play an important role in problems involving Tauberian conditions, examples including the previously mentioned work of Hagelstein and Stokolos \cite{hs} and Hagelstein, Luque, and Parissis \cite{hlp}.   This naturally leads us to the following conjecture involving convex density bases. (Recall that a \emph{density basis} $\mathcal{B}$ in $\mathbb{R}^n$ is a collection of sets for which
$$\lim_{\substack{x \in R \in B \\ \operatorname{diam}(R) \rightarrow 0}}\frac{1}{|R|}\int_R \chi_E = \chi_E (x)$$ holds for a.e. $x\in\R^n$, for every set $E \subset \mathbb{R}^n$ of finite measure.   An important result of Busemann and Feller is that the maximal operator $M_{\mathcal{B}}$ associated to a homothecy invariant density basis $\mathcal{B}$ satisfies a Tauberian condition with respect to $\alpha$ for every $\alpha > 0$.   See \cites{BF, Gu} for details.)

\begin{conjecture}\label{conj.cont} Let $\mathcal{B}$ be a homothecy invariant density basis of bounded convex sets in $\mathbb{R}^n$.  Then the associated Tauberian constants $C_{\mathcal{B}}(\alpha)$ satisfy
$$
\lim_{\alpha \rightarrow 1^-}C_{\mathcal{B}}(\alpha) = 1.
$$
\end{conjecture}
The following theorem provides some evidence that the above conjecture is on the right track.

\begin{thm}\label{t.singlesol} Let $\mathcal{B}$ be a homothecy invariant density basis of convex sets in $\mathbb{R}^n$.   Then
$$
\big|\big\{x \in \mathbb{R}^n :\, M_{\mathcal{B}}\chi_E(x) = 1\big\}\big| = |E|
$$
holds for every measurable set $E$ in $\mathbb{R}^n$.
\end{thm}

To appreciate the role that convexity plays in the following argument, observe that the conclusion of this theorem does \emph{not} hold when $\mathcal{B}$ is the homothecy invariant collection of sets indicated in Example~\ref{exa.nonconvex} above.

\begin{proof}

 Let us fix some measurable set $E\subset\mathbb \mathbb R^n$ with $|E|>0$. Since $\mathcal{B}$ is a density basis, for a.e. $x \in \mathbb{R}^n$ we have that
$$
 \lim_{j \rightarrow \infty}\frac{1}{|R_{x,j}|}\int_{R_{x,j}}\chi_E= \chi_E(x),
$$
where $R_{x,j}$ is any sequence of sets in $\mathcal{B}$ containing $x$ whose diameters tend to $0$; for this and other basic properties of density bases, see \cite{Gu}*{Ch. III}. So
 $$
 E \subset  \{x \in \mathbb{R}^n:\, M_{\mathcal{B}}\chi_E(x) = 1\} \quad\text {a.e.}
 $$
 and in particular
\begin{equation}\label{e.singlesol}
 |E| \leq  | \{x \in \mathbb{R}^n :\, M_{\mathcal{B}}\chi_E(x) = 1 \} | .
\end{equation}
If $|E| = \infty$ the theorem automatically holds so we may assume without loss of generality that $|E| < \infty$.

The rest of the proof is by way of contradiction and the argument is divided into two basic steps.
	
\subsubsection*{Step 1:} Suppose that \eqref{e.singlesol} fails. Then there exists a set $A\subset E^\mathtt{c}$ with $|A|>0$ such that, for every $x\in A$ there exists a sequence of sets $\{R_{x,j}\}_j \subset \mathcal B$ satisfying $x\in R_{x,j}$ for all $j$, $\lim_{j\to +\infty}\operatorname{diam}(R_{x,j})=+\infty$ and
\begin{equation}\label{e.averageto1}
 \frac{1}{|R_{x,j}|}\int_{R_{x,j}} \chi_E>1-\frac{1}{j},\quad j=2,3,\ldots \, .
\end{equation}
We now prove this claim. Assuming that \eqref{e.singlesol} fails and letting
$$
\mathcal H_E\coloneqq \{x \in \mathbb{R}^n :\, M_{\mathcal{B}}\chi_E(x) = 1 \}\setminus E
$$
we have that $|\mathcal H_E|>0$. Now let $A$ denote the set
$$
 A= \mathcal H_E \cap\bigg\{x\in E^\mathtt{c}: \lim_{\substack {x\in R\in\mathcal B\\ \operatorname{diam}(R)\to 0}}\frac{1}{|R|}\int_{R}\chi_E =0 \bigg\}.
$$
Since $\mathcal B$ is a density basis we have that $|A|=|\mathcal H_E|>0$. We fix $x\in A$. Since $x\in \mathcal H _E $ we conclude that for every positive integer $j\geq 2$ there exists a sequence $\{R_{x,j}\}_j\subset \mathcal B$, $x\in R_{x,j}$ for each $j$ and \eqref{e.averageto1} holds. It remains to show that $\lim_{j\to +\infty}\operatorname{diam}(R_x,j)=+\infty$. By the definition of $A$ there exists $\delta=\delta_x>0$ such that
$$
x\in R\in\mathcal B, \, \operatorname{diam}(R)<\delta \Rightarrow \frac{1}{|R|}\int_R \chi_E <\frac{1}{2}.
$$
Furthermore, it is clear that $\inf_j \operatorname{diam}(R_{x,j})\geq c>0$ otherwise the averages in \eqref{e.averageto1} would have a subsequence converging to $0$. The previous discussion and the convexity hypothesis for the collection $\mathcal B$ imply that there exists a homothetic copy $S_{R_j}$ of $R_{x,j}$ with $\operatorname{diam}(S_{R_j})=\frac{1}{2}\min(c,\delta)$ that satisfies
$$
x\in S_{R_j}\subset R_{x,j} \quad\text{and}\quad \frac{|E\cap S_{R_j}|}{|S_{R_j}|}<\frac{1}{2}.
$$
It is essential to notice here that the diameter of $S_{R_j}$ is independent of $j$. We have
\begin{align*}
1-\frac{1}{j}\leq \frac{|E\cap R_{x,j}|}{|R_{x,j}|}&=\frac{|E\cap S_{R_j} |}{|R_{x,j}|}+ \frac{|E\cap R_{x,j}\setminus S_{R_j} |}{|R_{x,j}|}
\\
&\leq \frac{|E\cap S_{R_j} |}{|R_{x,j}|}+\frac{|R_{x,j}|-|S_{R_j}|}{|R_{x,j}|}
\\
&= \frac{|E\cap S_{R_j}|}{|S_{R_j}|}\bigg(\frac{\operatorname{diam}(S_{R_j})}{\operatorname{diam}(R_{x,j})}\bigg)^n+1-\bigg(\frac{\operatorname{diam}(S_{R_j})}{\operatorname{diam}(R_{x,j})}\bigg)^n
\\
&\leq 1-\frac{1}{2}\bigg(\frac{\operatorname{diam}(S_{R_j})}{\operatorname{diam}(R_{x,j})}\bigg)^n.
\end{align*}
Thus we have
$$
\operatorname{diam}(R_{x,j}) \geq \frac{\operatorname{diam}(S_{R_j}) }{2^\frac{1}{n}}j^\frac{1}{n}=\frac{\frac{1}{2}\min(c,\delta) }{2^\frac{1}{n}}j^\frac{1}{n} \to +\infty\quad\text{as}\quad j\to+\infty.
$$
This proves the claim of the first step.

\subsubsection*{Step 2:} Suppose that $\{R_j\}_j$ is a sequence of convex sets whose diameters satisfy $\operatorname{diam}(R_j)\to+\infty$ and $\sup_j |R_j|<+\infty$. Then for any bounded set $B$ we have that $\lim_{j\to+\infty}|B\cap R_j|= 0$.

To see this note that every convex set in $\mathbb R^n$ is contained in a rectangle of comparable volume. Thus we can assume that $\{R_j\}_j$ is a sequence of rectangles in $\mathbb R^n$. Since $\sup_j |R_j|<+\infty$ and the diameters of the rectangles $R_j$ tend to infinity we conclude that there is a one-dimensional side $I_j$ of $R_j$ such that $\lim_{j\to+\infty}|I_j|=0$. The claim now follows since
$$
|R_j\cap B|\leq |I_j| |\operatorname{diam}(B)|^{n-1}\to 0\quad\text{as}\quad j\to+\infty.
$$

We can now conclude the proof of theorem. Assuming~\eqref{e.singlesol} does not hold let us consider the set $A$ provided by the first step above. We fix some ball $B(0,r)$ and $x\in A$ and $R_{x,j}\ni x$ as in the first step.  Note that, necessarily, $\sup_j|R_{x,j}|<+\infty$ because of the validity of \eqref{e.averageto1}. Thus
$$
\frac{|B(0,r)^\mathtt{c} \cap E\cap R_{x,j}|}{|R_{x,j}|}\geq \frac{ |E\cap R_{x,j}|}{|R_{x,j}|}-\frac{|B(0,r)\cap R_{x,j}|}{|R_{x,j}|}\to 1\quad\text{as}\quad j\to+\infty
$$
by \eqref{e.averageto1} and the statement of the second step. This implies that for any $r>0$ and $0<\lambda<1$ we have
$$
A\subset \{x\in \mathbb R^n: M(\chi_{E\cap B(0,r) ^\mathtt{c}})>\lambda\}.
$$
However $\mathcal B$ is a homothecy invariant density basis so by the Tauberian condition we should have
$$
0<|A|\leq  |\{x\in \mathbb R^n: M(\chi_{E\cap B(0,r) ^\mathtt{c}})>\lambda\}|\leq c(\lambda)|E\cap B^\mathtt{c}(0,r)|
$$
which is clearly a contradiction since $|E|<+\infty$ and thus $|E\cap B(0,r)^\mathtt{c}|\to 0$ as $r\to+\infty$.
\end{proof}

We are quickly exhausting all that we know at the moment regarding Solyanik estimates in harmonic analysis. As a closing remark, it is worth noting that Theorem~\ref{t.singlesol} provides a viable strategy to proving Conjecture~\ref{conj.cont}. Namely, to prove Conjecture~\ref{conj.cont} it now suffices to prove the following:

\begin{conjecture} Let $\mathcal{B}$ be a homothecy invariant density basis of convex sets in $\mathbb{R}^n$.  Suppose for some $\gamma > 1$ we have that, for every $0 < \alpha < 1$, there exists a set $E_{\alpha, \gamma}$ such that
$$
\big|\big\{x \in \mathbb{R}^n :\, M_{\mathcal{B}}\chi_{E_{\alpha, \gamma}}(x) > \alpha\big\}\big| \geq \gamma |E_{\alpha, \gamma}|.
$$
Then there exists a set $E_{\gamma}$ and a constant $c(\gamma)>1$ such that
$$
\big|\big\{x \in \mathbb{R}^n :\, M_{\mathcal{B}}\chi_{E_{\gamma}}(x) = 1\big\}\big| \geq c(\gamma) |E_{\gamma}|.
$$
\end{conjecture}

%%%%%%%%%%%%%%%%%%%%%%%%%%%%%% SECTION  SECTION SECTION

%%%%%%%%%%%%%%%%%%%%%%%%%%%%%% SECTION  SECTION SECTION

\end{document}